\documentclass[11pt,leqno,draft]{article}
\usepackage{amsfonts}
\usepackage{amsmath, amsthm, amsfonts, amssymb, color}
\usepackage{mathrsfs}
\usepackage{color}
\newtheorem{thm}{Theorem}[section]

\newtheorem{lem}[thm]{Lemma}

\newtheorem{rem}[thm]{Remark}
\theoremstyle{definition}
\newtheorem{defn}{Definition}[section]
\newcommand{\scr}[1]{\mathscr #1}
\definecolor{wco}{rgb}{0.5,0.2,0.3}

\numberwithin{equation}{section} \theoremstyle{remark}

\newcommand{\ua}{\uparrow}

\title{{\bf   Approximations of  McKean-Vlasov SDEs with Irregular Coefficients}\footnote{Supported in
 part by  NNSFC (11801406).} }
\author{
{\bf  Jianhai Bao$^{b)}$,  Xing Huang$^{a)}$}\\
\footnotesize{$^{a)}$Center for Applied Mathematics, Tianjin
University, Tianjin 300072, China}\\
\footnotesize{  xinghuang@tju.edu.cn}\\
\footnotesize{$^{b)}$Department of Mathematics, Swansea University,
Singleton Park, SA2 8PP, UK}\\
\footnotesize{Jianhai.Bao@Swansea.ac.uk}}
\begin{document}
\allowdisplaybreaks
\def\R{\mathbb R}  \def\ff{\frac} \def\ss{\sqrt} \def\B{\mathbf
B} \def\W{\mathbb W}
\def\N{\mathbb N} \def\kk{\kappa} \def\m{{\bf m}}
\def\ee{\varepsilon}\def\ddd{D^*}
\def\dd{\delta} \def\DD{\Delta} \def\vv{\varepsilon} \def\rr{\rho}
\def\<{\langle} \def\>{\rangle} \def\GG{\Gamma} \def\gg{\gamma}
  \def\nn{\nabla} \def\pp{\partial} \def\E{\mathbb E}
\def\d{\text{\rm{d}}} \def\bb{\beta} \def\aa{\alpha} \def\D{\scr D}
  \def\si{\sigma} \def\ess{\text{\rm{ess}}}
\def\beg{\begin} \def\beq{\begin{equation}}  \def\F{\scr F}
\def\Ric{\text{\rm{Ric}}} \def\Hess{\text{\rm{Hess}}}
\def\e{\text{\rm{e}}} \def\ua{\underline a} \def\OO{\Omega}  \def\oo{\omega}
 \def\tt{\tilde} \def\Ric{\text{\rm{Ric}}}
\def\cut{\text{\rm{cut}}} \def\P{\mathbb P} \def\ifn{I_n(f^{\bigotimes n})}
\def\C{\scr C}      \def\aaa{\mathbf{r}}     \def\r{r}
\def\gap{\text{\rm{gap}}} \def\prr{\pi_{{\bf m},\varrho}}  \def\r{\mathbf r}
\def\Z{\mathbb Z} \def\vrr{\varrho} \def\ll{\lambda}
\def\L{\scr L}\def\Tt{\tt} \def\TT{\tt}\def\II{\mathbb I}
\def\i{{\rm in}}\def\Sect{{\rm Sect}}  \def\H{\mathbb H}
\def\M{\scr M}\def\Q{\mathbb Q} \def\texto{\text{o}} \def\LL{\Lambda}
\def\Rank{{\rm Rank}} \def\B{\scr B} \def\i{{\rm i}} \def\HR{\hat{\R}^d}
\def\to{\rightarrow}\def\l{\ell}\def\iint{\int}
\def\EE{\scr E}\def\Cut{{\rm Cut}}
\def\A{\scr A} \def\Lip{{\rm Lip}}
\def\BB{\scr B}\def\Ent{{\rm Ent}}\def\L{\scr L}
\def\R{\mathbb R}  \def\ff{\frac} \def\ss{\sqrt} \def\B{\mathbf
B}
\def\N{\mathbb N} \def\kk{\kappa} \def\m{{\bf m}}
\def\dd{\delta} \def\DD{\Delta} \def\vv{\varepsilon} \def\rr{\rho}
\def\<{\langle} \def\>{\rangle} \def\GG{\Gamma} \def\gg{\gamma}
  \def\nn{\nabla} \def\pp{\partial} \def\E{\mathbb E}
\def\d{\text{\rm{d}}} \def\bb{\beta} \def\aa{\alpha} \def\D{\scr D}
  \def\si{\sigma} \def\ess{\text{\rm{ess}}}
\def\beg{\begin} \def\beq{\begin{equation}}  \def\F{\scr F}
\def\Ric{\text{\rm{Ric}}} \def\Hess{\text{\rm{Hess}}}
\def\e{\text{\rm{e}}} \def\ua{\underline a} \def\OO{\Omega}  \def\oo{\omega}
 \def\tt{\tilde} \def\Ric{\text{\rm{Ric}}}
\def\cut{\text{\rm{cut}}} \def\P{\mathbb P} \def\ifn{I_n(f^{\bigotimes n})}
\def\C{\scr C}      \def\aaa{\mathbf{r}}     \def\r{r}
\def\gap{\text{\rm{gap}}} \def\prr{\pi_{{\bf m},\varrho}}  \def\r{\mathbf r}
\def\Z{\mathbb Z} \def\vrr{\varrho} \def\ll{\lambda}
\def\L{\scr L}\def\Tt{\tt} \def\TT{\tt}\def\II{\mathbb I}
\def\i{{\rm in}}\def\Sect{{\rm Sect}}  \def\H{\mathbb H}
\def\M{\scr M}\def\Q{\mathbb Q} \def\texto{\text{o}} \def\LL{\Lambda}
\def\Rank{{\rm Rank}} \def\B{\scr B} \def\i{{\rm i}} \def\HR{\hat{\R}^d}
\def\to{\rightarrow}\def\l{\ell}
\def\8{\infty}\def\I{1}\def\U{\scr U}
\maketitle

\begin{abstract}
 The goal of this paper is to approximate several kinds of {\it
McKean-Vlasov SDEs} with {\it irregular
 coefficients} via weakly interacting particle systems.  More precisely,  propagation of chaos and  convergence rate of Euler-Maruyama
 scheme associated with  the consequent weakly interacting
 particle systems  are
 investigated for McKean-Vlasov SDEs, where (i) the
diffusion terms are H\"older continuous by taking advantage of
Yamada-Watanabe's approximation
 approach and (ii) the drifts  are H\"older continuous by freezing distributions followed by invoking Zvonkin's
 transformation trick.

\end{abstract} \noindent
 AMS subject Classification:\  65C05, 65C30, 65C35.   \\
\noindent
 Keywords: McKean-Vlasov SDE,  Yamada-Watanabe
 approximation, Zvonkin's transformation, H\"{o}lder continuity
 \vskip 2cm

\section{Introduction and Main Results}
The pioneer work on McKean-Vlasov SDEs whose coefficients are
dependent on   laws of the solutions is initiated in \cite{Mc}. In
terminology, McKean-Vlasov SDEs are also referred to as
distribution-dependent SDEs or mean-field   SDEs, which are derived
as a limit of interacting diffusions. Since McKean's work,
McKean-Vlasov SDEs have been applied extensively in   stochastic
control, queue systems, mathematical finance, multi-factor
stochastic volatility and hybrid models, to name a few; see, for
example, \cite{BLM,CD}. So far, McKean-Vlasov SDEs have been
investigated considerably e.g. on wellposedness \cite{CD,SZ},
ergodicity \cite{EGZ,W19}, Feyman-Kac Formulae \cite{BLPR,CM,W19},
Harnack inequalities \cite{HW,Wangb}.

In general, McKean-Vlasov SDEs cannot be solved explicitly so it is
desirable  to devise implementable numerical algorithms so that they
can be simulated. With contrast to the standard SDEs, the primary
challenge to simulate McKean-Vlasov SDEs lies in approximating the
distributions at each step. At present, there exist a few of results
on numerical approximations for McKean-Vlasov SDEs; see e.g.
\cite{ABRS,BF,CST,DEG,Mal,ST}. In particular, \cite{DEG} is
concerned with strong convergence of tamed Euler-Maruyama (EM for
short) scheme for McKean-Vlasov SDEs, where the drift terms are of
superlinear growth, and \cite{CST,ST} are devoted to weak
convergence for EM algorithms.  The strong convergence of numerical
algorithms for {\it McKean-Vlasov} SDEs with {\it irregular
coefficients} is rather scarce although there are plenty of results
on convergence of numerical approximations for standard SDEs with
irregular coefficients, see e.g. \cite{BHY,GR,NT}. Nevertheless, in
the present work we intend to go further and aim to investigate
strong convergence of EM scheme associated with several class of
{\it McKean-Vlasov SDEs with irregular coefficients}.

Next we start with some notations. Let $\mathscr{P}(\R^d)$ be the
collection of all probability measures on $\R^d$. For   $p>0$, if
$\mu\in\mathscr{P}(\R^d)$ enjoys finite $p$-th moment, i.e.,
$\mu(|\cdot|^p):=\int_{\R^d}|x|^p\mu(\d x)<\8$, we then formulate
$\mu\in\mathscr{P}_p(\R^d)$.  For $\mu,\nu\in\mathscr{P}_p(\R^d)$,
$p>0,$ the $\mathbb{W}_p$-Wasserstein distance between $\mu$ and
$\nu$ is defined by
\begin{equation*}
\mathbb{W}_p(\mu,\nu)=\inf_{\pi\in\mathcal
{C}(\mu,\nu)}\Big(\int_{\R^d\times\R^d}|x-y|^p\pi(\d x,\d y)\Big)^{\ff{1}{1\vee
p}},
\end{equation*}
where $\mathcal {C}(\mu,\nu)$ stands for the set of all couplings of
$\mu$ and $\nu$. Let $\delta_x$ be  Dirac's delta measure centered
at the point $x\in\R^d.$ As for a random variable $\xi,$ its law is
written by $\mathscr{L}_\xi$. For any $t\ge0,$ let $C([0,t];\R^d)$
be the set of all continuous functions $f:[0,t]\to\R^d$ endowed with
the uniform norm $\|f\|_{\8,t}:=\sup_{0\le s\le t}|f(s)|$. $\lfloor
a\rfloor$ stipulates the integer part of $a\ge0.$

Consider the following McKean-Vlasov SDE on $\R$ \beq\label{E1} \d
X_t= b(X_t, \mu_t)\d t +\si(X_t)\d W_t,~~~t\ge0,~~~X_0=\xi,
\end{equation}
where $\mu_t:=\mathscr{L}_{X_t}$ stands for the law of $X$ at time
$t,$ $b:=b_1+b_2,$ $b_i:\R\times\scr P(\R)\rightarrow\R$, $i=1,2$,
$\si:\R\rightarrow\R$, and $(W_t)_{t\ge0}$ is a $1$-dimensional
Brownian motion on a complete filtration probability space
$(\OO,\F,(\F_t)_{t\ge0},\P)$.

Now we introduce the definition of strong solution to \eqref{E1},
which is standard in literature; see e.g. \cite[Definition
1.1]{Wangb}.

 {\begin{defn}\label{def1} A continuous adapted process $(X_t)_{t\geq
0}$ on $\mathbb{R}$ is called a (strong) solution of \eqref{E1}, if
$$\int_0^t\E\left(|b(X_s,\mu_s)|+|\si(X_s)|^2\right)\d s<\infty, \ \ t\geq 0,$$
and $\P$-a.s.
$$X_t=X_0+\int_0^tb(X_s,\mu_s)\d s+\int_0^t\si(X_s)\d W_s, ~~~~t\ge0.$$
\end{defn}
\begin{rem}\label{Rem1}
By BDG's inequality, Definition \ref{def1} yields $\E\Big(\sup_{0\le
s\le t}|X_s|\Big)<\infty$ if $\E|X_0|<\infty$.
\end{rem}

With regard to the coefficients of \eqref{E1}, we assume that
\begin{enumerate}
\item[({\bf H1})]  For fixed $\mu\in\scr P(\R)$, $x\mapsto b_1(x,\mu)$ is continuous and non-increasing,
and there exist   $K_1>0$    and $ \beta\in(0,1]$ such that, for
$x,y\in\R$ and $\mu,\nu\in\scr P_1(\R)$,
\begin{equation}\label{b1}
    |b_1(x,\mu)-b_1(x,\nu)|\le K_1\W_1(\mu,\nu),  \  |b_1(x,\mu)-b_1(y,\mu)|\leq
    K_1|x-y|^\bb,
\end{equation}
\begin{equation}\label{EQ1}
|b_2(x,\mu)-b_2(y,\nu)|\le K_1(|x-y|+\mathbb{W}_1(\mu,\nu)).
\end{equation}

\item[({\bf H2})] There exist constants
$K_2>0$ and $\aa\in[\ff{1}{2},1]$ such that  $ |\si(x)-\si(y)|\le
K_2|x-y|^{\alpha}, x,y\in\R. $
\end{enumerate}

The theorem below addresses the strong wellposedness of \eqref{E1}.

\begin{thm} \label{Existence} Assume that {\bf(H1)} and {\bf(H2)} hold. Then, for   $X_0=\xi\in\F_0$ with $\L_\xi\in\scr P_p(\R)$, $p\ge2$, \eqref{E1}
has a unique strong solution $(X_t^\xi)_{t\ge0}$ with the initial
value $X_0^\xi=\xi$ such that
\begin{equation}\label{eq3}
\E\Big(\sup_{0\le t\le T}|X_t^\xi|^p\Big)\le C_T(1+\E|\xi|^p)
\end{equation}
for some constant $C_T>0.$
\end{thm}

Existence and uniqueness of McKean-Vlasov SDEs with regular
coefficients have been investigated extensively; see e.g.
\cite{BMP,CD,MV,SZ, Wangb}. Meanwhile, the strong wellposedness of
McKean-Vlasov SDEs with irregular coefficients has also received
much attention; see, for example, \cite{Ch,HW}, where, in \cite{Ch},
the dependence of laws  is of integral type and the diffusion is
non-degenerate, and \cite{HW} is concerned with the integrability
condition but excluding   linear growth of the drift. For weak
wellposedness of McKean-Vlasov SDEs, we refer to e.g.
\cite{HW,LM,MS,MV}. Whereas Theorem \ref{Existence} shows that the
McKean-Vlasov SDE  we are interested in is strongly wellposed
although both the drift term and the diffusion term are irregular in
certain sense.

Since \eqref{E1} is distribution-dependent, we exploit the
stochastic interacting particle systems to approximate it. Let
$N\ge1$ be an integer and $(X_0^i,W^i_t)_{1\le i\le N}$ be i.i.d.
copies of $(X_0,W_t).$ Consider the following stochastic
non-interacting particle systems
\begin{equation}\label{C4}
\d X_t^i=b(X_t^i,\mu_t^i)\d t+\si(X_t^i)\d
W_t^i,~~~t\ge0,~~~i\in\mathcal {S}_N:=\{1,\cdots,N\}
\end{equation}
with $\mu_t^i:=\mathscr{L}_{X_t^i}$. By the weak uniqueness due to
Theorem \ref{Existence}, we have $\mu_t=\mu_t^i, i\in\mathcal
{S}_N.$ Let $ \tt\mu_t^N $ be the empirical distribution associated
with $X_t^1,\cdots,X_t^N$, i.e.,
\begin{equation}\label{H1}
\tt\mu_t^N =\ff{1}{N}\sum_{j=1}^N\dd_{X_t^j}.
\end{equation} Moreover, we need
to consider the so-called stochastic $N$-interacting particle
systems:
\begin{equation}\label{eq4}
\d X_t^{i,N}=b(X_t^{i,N},\hat\mu_t^N)\d t+\si(X_t^{i,N})\d
W_t^i,~t\ge0,~X_0^{i,N}=X_0^i,~i\in\mathcal {S}_N,
\end{equation}
where  $\hat\mu_t^N$ means the empirical distribution corresponding
to  $X_t^{1,N},\cdots,X_t^{N,N}$, namely,
\begin{equation*}
 \hat\mu_t^N :=\ff{1}{N}\sum_{j=1}^N\dd_{X_t^{j,N}}.
 \end{equation*}
We remark that particles $(X^i)_{i\in\mathcal {S}_N}$ are mutually
independent and that particles $(X^{i,N})_{i\in\mathcal {S}_N}$ are
interacting and are not independent. Furthermore, under {\bf(H1)}
and {\bf(H2)}, the stochastic $N$-interacting particle systems
\eqref{eq4} are strongly wellposed; see Lemma \ref{lem} below for
more details.

To discretize \eqref{eq4} in time, we introduce the continuous time
EM scheme defined as below: for any  $\dd\in(0,\e^{-1}),$
\begin{equation}\label{C3}
\d X_t^{\dd,i,N}=b(X_{t_\dd}^{\dd,i,N},\hat\mu_{t_\dd}^{\dd,N})\d
t+\si(X_{t_\dd}^{\dd,i,N})\d
W_t^i,~~~t\ge0,~~~X_0^{\dd,i,N}=X_0^{i,N},
\end{equation}
where $t_\dd:=\lfloor t/\dd\rfloor\dd$  and
 $$\hat\mu_{k\dd}^{\dd,N} :=\ff{1}{N}\sum_{j=1}^N\dd_{X_{k\dd}^{\dd,j,N}},~~~~k\ge0.$$

The following result states  that the continuous time EM scheme
corresponding to stochastic interacting particle systems converges
strongly to the non-interacting particle system whenever  the
particle number goes to infinity and the stepsize approaches to zero
and moreover provides the convergence rate.

\begin{thm}\label{th2}
 Assume {\bf(H1)}
and {\bf(H2)} hold and suppose further $\mathscr{L}_{X_0}\in  \scr
P_p(\R)$ for some $p>4$.  Then, for any $T>0$, there exists a
constant $C_T>0$ such that
\begin{equation}\label{F3}
\sup_{i\in\mathcal {S}_N}\E\Big(\sup_{0\le t\le
T}|X_t^i-X_t^{\dd,i,N}|\Big)\le C_T
\begin{cases}
N^{-\ff{1}{8}}+
 \Big(\ff{1}{ \ln\ff{1}{\dd}}\Big)^{1/2},~~~~~~~~~~~~~~~~~~\aa=\ff{1}{2}\\
 N^{-\ff{2\aa-1}{4}}+ \dd^{\ff{(2\aa-1)^2}{2}}+\dd^{\ff{\bb(2\aa-1)}{2}}
 ,~~~\aa\in(\ff{1}{2},1]
\end{cases}
\end{equation}
and
\begin{equation}\label{F4}
\sup_{i\in\mathcal {S}_N}\E\Big(\sup_{0\le t\le
T}|X_t^i-X_t^{\dd,i,N}|^2\Big)\le C_T
\begin{cases}
N^{-\ff{1}{4}}+ \ff{1}{ \ln\ff{1}{\dd} },~~~~~~~~~~~~~\aa=\ff{1}{2}\\
N^{-\ff{1}{4}}+\dd^{2\aa-1}+\dd^\bb,~~~~ \aa\in(\ff{1}{2},1)\\
N^{-\ff{1}{4}}+\dd+\dd^\bb,~~~~~~~~~~ \aa=1.
\end{cases}
\end{equation}
\end{thm}

The assumption  on the $p$-th moment of the initial value is set to
ensure that Glivenko-Cantelli convergence under the Wasserstein
distance (see e.g. \cite[Theorem 5.8]{CD})  is available. According
to Theorem \ref{th2}, it is preferable to measure the convergence
between the non-interacting particle systems and the continuous time
EM scheme of the corresponding stochastic interacting particle
systems in a lower order moment. Moreover, Theorem \ref{th2} extends
\cite{BHY,GR} to McKean-Vlasov SDEs with H\"older continuous
diffusions.

In the preceding section, we focus mainly  on McKean-Vlasov SDEs,
where, in particular,  the diffusion term is H\"older continuous. We
now move forward to consider McKean-Vlasov SDEs, in which the drift
coefficients are allowed to be H\"older continuous w.r.t. the
spatial variables and Lipschitz in law. In the sequel, we are still
interested in \eqref{E1} but for the multidimensional setting. More
precisely, for $d\geq 1$, we work on the following McKean-Vlasov SDE
\begin{equation}\label{C2}
\d X_t=b(X_t,\mu_t)\d t+\si(X_t)\d W_t, ~~~~t\ge0, ~~X_0=\xi,
\end{equation}
where $b:\R^d\times\mathscr{P}(\R^d)\to\R^d,$
$\si:\R^d\to\R^d\otimes\R^d$, and $(W_t)_{t\ge0}$ is a
$d$-dimensional Brownian motion on some complete filtration  probability space
$(\OO,\F,(\F_t)_{t\ge0},\P)$.

Concerning \eqref{C2}, we assume that for any $x,y\in\R^d$ and
$\mu,\nu\in\mathscr{P}_1(\R^d),$
\begin{enumerate}
\item[{\bf (A1)}]  $\si\in C^1(\R^d;\R^d\otimes\R^d)$,  $\si(x)$ is invertible, and
\begin{equation}\label{A5}
\|b\|_\8+\|\si \|_\8+\|\nn\si \|_\8+\|\si^{-1}\|_\8<\8,
\end{equation}
where $\nn$ denotes   gradient operator.
\item[{\bf (A2)}] There exist   constants $K>0,\aa\in(0,1]$ such that
\begin{equation}\label{C1}
|b(x,\mu)-b(y,\nu)|\le K\{|x-y|^\aa+\mathbb{W}_1(\mu,\nu)\}.
\end{equation}
\end{enumerate}

Note that, by ({\bf A2}),  the drift $b$ is at most of linear
growth, i.e., there exists a constant $C>0$ such that
\begin{equation*}
|b(x,\mu)|\le
C(1+|x|+\mathbb{W}_1(\mu,\dd_{{\bf0}})),~~~~x\in\R^d,~~\mu\in\mathscr{P}_1(\R^d),
\end{equation*}
and that the diffusion $\si$ is uniformly bounded and nondegenerate
due to ({\bf A1}). Whence, by virtue of \cite[Theorems 2.3 \&
2.7]{BMP} (see also \cite[Theorem 3.2]{LM}), under ({\bf A1}) and
({\bf A2}), \eqref{C2} has a unique weak solution $(\bar\OO,
\bar\F,\bar\P,\bar W, \bar X)$. Let
$b_t^{\bar\mu}(x)=b(x,\bar\mu_t)$ with $\bar\mu_t:=\mathscr{L}_{\bar
 X_t}|\bar\P$, the law of $\bar X_t$ under $\bar\P$. On the other hand, in the light of
 \cite[Theorem 1.1]{Ch0} in the non-degenerate case, the
time-dependent SDE
\begin{equation}\label{SEU}
\d X_t= b_t^{\bar\mu}(X_t)\d t+\si(X_t)\d W_t,~~t\ge0
\end{equation}
admits a unique strong solution under ({\bf A1}) and ({\bf A2}).
 Hence, we conclude that \eqref{C2} enjoys a unique strong solution.
In fact, the weak existence of \eqref{C2} plus  the strong existence
and uniqueness of \eqref{SEU} implies the strong existence of
\eqref{C2} by \cite[Lemma 3.4]{HW}. By the weak uniqueness of
\eqref{C2}, the strong uniqueness of \eqref{C2} is equivalent to
that of \eqref{SEU}. By following the same trick in the proof of
Lemma \ref{lem}, we conclude that \eqref{eq4} is strongly wellposed
under ({\bf A1}) and ({\bf A2}). In what follows, we emphasize that
the stochastic $N$-interacting particle systems and the
corresponding EM scheme associated with \eqref{C2} still solve
\eqref{eq4} and \eqref{C3}, respectively, but for the
multidimensional setup.

Another contribution in present paper is concerned with strong
convergence between non-interacting particle systems and continuous
time EM scheme of stochastic interacting particle systems
corresponding to McKean-Vlasov SDEs, where the drift is singular
w.r.t. the spatial variable.

\begin{thm}\label{th3}
Assume $({\bf A1})$ and $({\bf A2})$ hold and  suppose further
$\mathscr{L}_{X_0}\in  \scr P_p(\R^d)$ for some $p>4$.  Then, for
any $T>0$, there exists a constant $C_T>0$ such that
\begin{equation}\label{A6}
\sup_{i\in\mathcal {S}_N}\E\Big(\sup_{0\le t\le
T}|X_t^i-X_t^{\dd,i,N}|^2\Big)\le C_T
\begin{cases}
N^{-\ff{1}{2}}+ \dd^\aa,~~~~~~~~~~~~~d<4\\
N^{-\ff{1}{2}}\log N+\dd^\aa,~~~~~~ d=4\\
N^{-\ff{2}{d}}+\dd^\aa,~~~~~~~~~~~~~d>4.
\end{cases}
\end{equation}
\end{thm}
\begin{rem}
In \eqref{C2}, we set $\si$  to be independent of distribution
variables merely to be consistent with the framework of \eqref{E1}.
Whereas, by examining  argument of Proposition \ref{pro3} below, the
diffusion term  can be allowed to be distribution-dependent as long
as it is Lipschitz in spatial argument and Lipschitz in law.
Moreover, we remark that, by the standard truncation argument (see
e.g. \cite{BHY}) and stopping time strategy (see e.g. \cite{DEG}),
the uniform boundedness of the drift $b$ can be removed.
\end{rem}

The remainder of this paper is arranged as follows: In Section
\ref{sec1}, the wellposedness of \eqref{E1}  is addressed by
Yamada-Watanabe's approximation; Section \ref{sec2} is devoted to
completing the proof of Theorem \ref{th2} via Yamada-Watanabe's
approach; The last section  aims to finish the proof of Theorem
\ref{th3} by employing Zvonkin's transformation.

\section{Proof of Theorem \ref{Existence}}\label{sec1}
To complete the proofs of Theorems \ref{Existence} and \ref{th2}, we
shall adopt the Yamada-Watanabe approximation approach (see e.g.
\cite{GR,IW}), where the essential ingredient  is to approximate the
function $\R\ni x\mapsto|x|$ in an appropriate manner. For $\gg>1$
and $\vv\in(0,1),$ one trivially has $\int_{\vv/\gg}^\vv\ff{1}{x}\d
x=\ln\gg$ so that there exists a continuous function
$\psi_{\gg,\vv}:\R_+\to\R_+$ with support $[\vv/\gg,\vv]$ such that
\begin{equation*}
0\le\psi_{\gg,\vv}(x)\le \ff{2}{x \ln\gg},~~~x>0, ~~\mbox{ and
}~~\int^\vv_{\vv/\gg} \psi_{\gg,\vv}(x)\d x=1.
\end{equation*}
By a direct calculation, the following mapping
\begin{equation}\label{F8}
\R\ni x\mapsto
V_{\gg,\vv}(x):=\int_0^{|x|}\int_0^y\psi_{\gg,\vv}(z)\d z\d y
\end{equation}
is $C^2$ and satisfies
\begin{equation}\label{R1}
|x|-\vv\le V_{\gg,\vv}(x)\le |x|,~~x\in\R,~~ V_{\gg,\vv}'(x)
\in[0,1], ~x\ge0, ~~~~V_{\gg,\vv}'(x) \in[-1,0], ~x<0
\end{equation}
and
\begin{equation}\label{R2}
 0\le V_{\gg,\vv}''(x)\le
\ff{2}{|x|\ln\gg}{\bf1}_{[\vv/\gg,\vv]}(|x|),~~~~x\in\R.
\end{equation}

With the function $V_{\gg,\vv}$, introduced in \eqref{F8}, in hand,
we are in position to complete
\begin{proof}[{\bf Proof of Theorem \ref{Existence}}]
Below, we fix the time terminal $T>0.$ To obtain  existence of a
solution to \eqref{E1}, for each $k\ge1$, we consider the following
distribution-iterated
 SDE
\begin{equation}\label{E2}
\d X^{(k)}_t=b(X^{(k)}_t,\mu_t^{(k-1)})\d t+\si(X^{(k)}_t)\d
W_t,~~X_t^{(0)}\equiv \xi,~~~t\in[0,T],
\end{equation}
where $\mu_t^{(k)}:=\mathscr{L}_{X^{(k)}_t},k\ge0$. For each fixed
$k\ge1$, according to \cite[Theorem 3.2]{IW}, \eqref{E2} has a
unique solution $(X_t^{(k)})_{t\ge0}$. Moreover,  by a standard
calculation, from {\bf(H1)}, {\bf (H2)} and the fact $\xi\in\scr
P_p(\mathbb{R})$ for some $p\geq 2$, we derive that
\begin{align}\label{EE0}
\sup_{k\geq 1}\E\Big(\sup_{0\le t\le T}|X_t^{(k)}|^2\Big)<\infty.
\end{align}
 For notation brevity, we set
$Z_t^{(k)}:=X^{(k)}_t-X^{(k-1)}_t$ and
$V_\vv:=V_{\e^{\ff{1}{\vv}},\vv}$ (that is, we herein take
$\gg=\e^{\ff{1}{\vv}})$. By It\^o's formula, for any $\ll\ge0$, it
follows that
\begin{equation}\label{EE}
\begin{split}
 \e^{-\ll t}V_\vv(Z_t^{(k+1)})&=-\ll \int_0^t\e^{-\ll s}V_\vv(Z_s^{(k+1)})\d s\\&\quad+\int_0^t\e^{-\ll s}
V'_\vv(Z_s^{(k+1)})\{b(X^{(k+1)}_s,\mu_s^{(k)})-b(X^{(k)}_s,\mu_s^{(k-1)})\} \d s\\
&\quad+\frac{1}{2}\int_0^t\e^{-\ll s}
V''_\vv(Z_s^{(k+1)})(\si(X^{(k+1)}_s)-\si(X^{(k)}_s))^2 \d s\\
&\quad+\int_0^t\e^{-\ll s} V'_\vv(Z_s^{(k+1)}) \{\si(X^{(k+1)}_s)-\si(X^{(k)}_s) \}\d W_s\\
&=:
I_{1,\vv}^\ll(t)+I_{2,\vv}^\ll(t)+I_{3,\vv}^\ll(t)+I_{4,\vv}^\ll(t).
\end{split}
\end{equation}
By virtue of \eqref{R1}, one obviously has
\begin{equation}
\begin{split}
I_{1,\vv}^\ll(t) \le \lambda \vv t-\ll\int_0^t\e^{-\ll s}
|Z_s^{(k+1)}|\d s.
\end{split}
\end{equation}
Furthermore, by   ({\bf H1}) and \eqref{R1}, we deduce  that
\begin{equation}\label{E3}
\begin{split}
I_{2,\vv}^\ll(t) &\le\int_0^t\e^{-\ll s}
V'_\vv(Z_s^{(k+1)})\{b_1(X^{(k+1)}_s,\mu_s^{(k)})-b_1(X^{(k)}_s,\mu_s^{(k)})\}\d
s\\
&\quad+\int_0^t\e^{-\ll s}
V'_\vv(Z_s^{(k+1)})\{b_1(X^{k}_s,\mu_s^{(k)})-b_1(X^{(k)}_s,\mu_s^{(k-1)})\}\d
s\\
&\quad+\int_0^t \e^{-\ll s} |
V'_\vv(Z_s^{(k+1)})|\cdot|b_2(X^{(k+1)}_s,\mu_s^{(k)})-b_2(X^{(k)}_s,\mu_s^{(k-1)})|\d
s\\
&\le 2K_1\int_0^t\e^{-\ll s}
\{|Z_s^{(k+1)}|+\mathbb{W}_1(\mu_s^{(k)},\mu_s^{(k-1)})\}\d s,
\end{split}
\end{equation}
where the first integral in the first inequality was dropped since,
for fixed $\mu\in\scr P(\R)$, $x\mapsto b_1(x,\mu)$ is
non-increasing. Next, by utilizing   ({\bf H2}) and \eqref{R2} with
$\gg=\e^{\ff{1}{\vv}}$ and using $\aa\in[1/2,1]$, we infer that
\begin{equation}\label{E4}
\begin{split}
| I_{3,\vv}^\ll(t)| \le\ff{K_2^2\vv}{2} \int_0^t\e^{-\ll s}
|Z_s^{(k+1)}|^{2\aa-1}{\bf1}_{[\ff{\vv}{\e^{1/\vv}},\vv]}(|Z_s^{(k+1)}|)
\d s   \le  \ff{1}{2}c_\ll K_2^2t\vv,
\end{split}
\end{equation}
where $c_\ll:=
\{1{\bf1}_{\{\ll=0\}}+\ff{1}{\ll}{\bf1}_{\{\ll>0\}}\}.$
So, taking advantage of \eqref{R1}, \eqref{E3} as well as \eqref{E4}
leads to
\begin{equation}\label{E7}
\begin{split}
  \e^{-\ll t}|Z_t^{(k+1)}|&\le 2(1+ (2\ll+c_\ll K_2^2)t)\vv -
(\ll-2K_1)\int_0^t\e^{-\ll s}
 |Z_s^{(k+1)}| \d
s\\
&\quad+2K_1\int_0^t \e^{-\ll s}\E|Z_s^{(k)}| \d s+I_{4,\vv}^\ll(t).
\end{split}
\end{equation}
 By \eqref{R1} and \eqref{EE0}, we have $\E I_{4,\vv}^\ll(t)=0$.
Whence,    choosing   $\ll=0$,  approaching $\vv\downarrow0$, and
employing Gronwall's inequality gives
\begin{equation}\label{W11}
 \E|Z_t^{(k+1)}|\le
 2K_1\e^{2K_1t}\int_0^t  \E|Z_s^{(k)}| \d
s.
\end{equation}
For notation simplicity, set
\begin{equation*}
|[Z^{(k)}]|_{\ll,t}:=\sup_{0\le s\le t}\Big(\e^{-\ll
s}\E|Z^{(k)}_s|\Big),~~~~\|[Z^{(k)}]\|_{\ll,t}:=\E\Big(\sup_{0\le
s\le t}(\e^{-\ll s}|Z^{(k)}_s|)\Big).
\end{equation*}
In the sequel,   we take $\ll\ge2K_1\e^{1+2K_1T}$ and let
$t\in[0,T].$ In terms of \eqref{W11}, it follows that
\begin{equation}\label{E0}
|[Z^{(k+1)}]|_{\ll,t}
\le \e^{-1} |[Z^{(k)}]|_{\ll,t}\le\e^{-k}|[Z^{(1)}]|_{\ll,T}.
\end{equation}
Subsequently, by invoking BDG's inequality, Jensen's inequality  and
\eqref{R1} and taking ({\bf H2}) into account followed by setting
$\vv\downarrow0$, we deduce from \eqref{E7} and $\aa\in[1/2,1]$ that
\begin{equation*}\label{eq2}
\begin{split}
\|[Z^{(k+1)}]\|_{\ll,t}  &\le2K_1 \int_0^t |[Z^{(k)}]|_{\ll,s} \d
s+{\bf 1}_{\{\aa=\ff{1}{2}\}}4\ss2K_2 \Big(\int_0^t
|[Z^{(k+1)}]|_{\ll,s} \d
s\Big)^{\frac{1}{2}}\\
&\quad+{\bf1}_{\{\aa\in(\ff{1}{2},1]\}}\Big\{\ff{1}{2}
\|[Z^{(k+1)}]\|_{\ll,t}+16K_2^2
\int_0^t|[Z^{(k+1)}]|_{\ll,s}^{2\aa-1}\d s\Big\}.
\end{split}\end{equation*}
This, in addition to \eqref{E0}, implies that there exists a
constant $C_T>0$ such that
\begin{equation*}
\begin{split}
\|[Z^{(k+1)}]\|_{\ll,t} \le C_T\,
\e^{- (\ff{1}{2}{\bf
1}_{\ff{1}{2}\{\aa=\ff{1}{2}\}}+(2\aa-1){\bf1}_{\{\aa\in(\ff{1}{2},1]\}})
k}.
\end{split}\end{equation*}
As a result, there exists an $\F_t$-adapted continuous stochastic
process $(X_t)_{t\in [0,T]}$ with $X_0=\xi$ and $\mu_t=\L_{X_t}$ such that
\begin{align}\label{cov}
\lim_{k\to\8}\sup_{t\in[0,T]}\mathbb{W}_1(\mu_t^{(k)},\mu_t)\le\lim_{k\to\infty}
\E\|X^{(k)}- X\|_{\8,t} =0.
\end{align}
From {\bf(H1)}, we infer that
\begin{equation*}\begin{split}
 \int_0^t |b(X_s^{(k)}, \mu_s^{(k-1)}) -b(X_s, \mu_s) |\d s
&\leq  \int_0^t |b_1(X_s^{(k)}, \mu_s) -  b_1(X_s,
\mu_s)| \d s \\
&\quad +2K_1\int_0^t\{|X_s^{(k)}-X_s|+\W_1(\mu_s^{(k-1)},\mu_s)\}\d
s.
\end{split}\end{equation*}
 By \eqref{b1},  and the continuity of $b_1(\cdot,\mu)$ for any $\mu\in \scr P_1(\R)$, we can apply \eqref{cov} and dominated convergence theorem to  obtain
\begin{equation}\label{M2}
\lim_{k\to\infty} \int_0^T \E|b(X_t^{(k)}, \mu_t^{(k-1)})  - b(X_t,
\mu_t)| \d t =0.
\end{equation} Again, from \eqref{cov} we find
that \begin{equation}\label{M1} \lim_{k\to\8}\E\Big(\sup_{0\le t\le
T}\Big|\int_0^t (\si(X_s^{(k)})  - \si(X_s)) \d W_s\Big|\Big)=0,
\end{equation} since, by
 {\bf(H2)}, BDG's inequality and Jensen's inequality, we have
\begin{equation*}
\begin{split}&\E\Big(\sup_{0\le t\le T}\Big|\int_0^t (\si(X_s^{(k)})  -
\si(X_s)) \d W_s\Big|\Big)\\&\leq
4\ss2K_2~\E\Big(\int_0^{T}|X_t^{(k)}-X_t|^{2\alpha} \d
t\Big)^{\frac{1}{2}}\\
&\le
{\bf1}_{\{\aa=\ff{1}{2}\}}4\ss2K_2\Big(\int_0^{T}\E|X_t^{(k)}-X_t|
\d t\Big)^{\frac{1}{2}}\\
&\quad+{\bf1}_{\{\aa\in(\ff{1}{2},1]\}}\Big\{\E\|X^{(k)}-X\|_{\8,T}+16K_2^2\int_0^T\E|X_t^{(k)}-X_t|^{2\aa-1}\d
t\Big\}.
\end{split}\end{equation*}
Now with \eqref{M2} and \eqref{M1} in hand, by   taking $k\to\infty$
in the equation
$$X_t^{(k)}= \xi+\int_0^t b(X_s^{(k)}, \mu_s^{(k-1)}) \d s +\int_0^t \si(X_s^{(k)}) \d W_s,\ \ k\ge 1, t\in [0,T],$$
  we derive (by extracting a suitable subsequence) $\P$-a.s.
\begin{equation*}
\d X_t=    b(X_t, \mu_t) \d t + \si(X_t) \d W_t, \ \ t\in
[0,T],
\end{equation*}
so that the existence of solution to \eqref{E1} is now available.

 Next, we prove the uniqueness of \eqref{E1}. To this end, we assume
that $(X_t^{1,\xi})_{t\ge0}$ and $(X_t^{2,\xi})_{t\ge0}$ are
solutions to \eqref{E1} with the same initial value $\xi.$ For
$\GG_t:=X_t^{1,\xi}-X_t^{2,\xi},$   by following the argument to
derive \eqref{W11},  one has
\begin{equation*}
 \E|\GG_t|\le
 2K_1\e^{2K_1t}\int_0^t  \E|\GG_s| \d
s,
\end{equation*}
which, by invoking  Gronwall's inequality and Remark \ref{Rem1},
yields  the uniqueness. Finally, we show that the $p$-th moment of
the solution process is uniformly bounded in a finite time interval.
Indeed, from \eqref{b1} and \eqref{EQ1}, there exists a constant $c
>0$ such that
\begin{equation}\label{W0}
|b(x,\mu)\le c\, \{1+|x|+\mathbb{W}_1(\mu,\dd_0)\}~~~\mbox{ and
}~~~|\si(x)|\le c\, (1+| x|),~~~~\mu\in\scr P_1(\R), \ \ x\in\R.
\end{equation}
Let $(X_t)_{t\ge0}$ be a  solution to \eqref{E1} with
$X_0=\xi\in\scr P_p(\R)$ for some $p\geq 2$. For any $n\geq 1$,  set
$\tau_{n}:=\inf\{t\geq 0, |X_t|\geq n\}$. Therefore, applying
H\"older's inequality and BDG's inequality yields
\begin{equation*}
\begin{split}
 \E\Big(\sup_{0\le s\le t\wedge \tau_n}|X_s|^p\Big)&\le 3^{p-1}\Big\{\E|\xi|^p+t^{p-1}\E\int_0^{t\wedge \tau _n}|b(X_s,\mu_s)|^p\d s+c_p\E\Big(\int_0^{t\wedge\tau_n}|\si(X_s)|^2\d
 s\Big)^{p/2}\Big\}\\
 &\le\ff{1}{2} \E\Big(\sup_{0\le s\le
 t\wedge\tau_n}|X_s|^p\Big)+C_t\int_0^t\E\Big(\sup_{0\le r\le
 s\wedge\tau_n}|X_r|^p\Big)\d s\\
 &+3^{p-1} \E|\xi|^p+C_t\int_0^t(\E|X_s|)^p\d s
\end{split}
\end{equation*}
for some positive increasing function $t\mapsto C_t$. Thus,
Gronwall's inequality and Remark \ref{Rem1} yield $\E\Big(\sup_{0\le
s\le t\wedge \tau_n}|X_s|^p\Big)<C_T(1+\E|\xi|^p)$ for some constant $C_T>0$ and for any $t\in[0,T]$, which
implies $\lim_{n\to\infty}\tau_n=\infty$. So, \eqref{eq3} holds by
Fatou's lemma.

\end{proof}

\section{Proof of Theorem \ref{th2}}\label{sec2}
In this section, we intend to finish the proof of Theorem \ref{th2}.
Before we start, we prepare some auxiliary materials.
 The lemma below address the wellposedness of the stochastic  $N$-interacting
 particle systems \eqref{eq4}.

\begin{lem}\label{lem}
 Assume that {\bf(H1)}   and {\bf(H2)}
hold. Then \eqref{eq4} admits a strong solution with
$$\sup_{i\in\mathcal {S}_N}\E\Big(\sup_{0\le t \le
T}|X_t^{i,N}|\Big)<\infty.$$
\end{lem}

\begin{proof}
For $x:=(x_1,\cdots,x_N)^*\in\R^N$,   $x_i\in\R$, set
\begin{equation*}
\begin{split}
\tt\mu^N_x:=\ff{1}{N}\sum_{i=1}^N\dd_{x_i}, ~~~\hat b(x):&=(b(x_1,\tt\mu^N_x),\cdots,b(x_N,\tt\mu^N_x))^*,~~~\\
 \hat\si(x):=\mbox{diag}(\si(x_1),\cdots,\si(x_N)),~~~~\hat W_t:&=(W_t^1,\cdots,W_t^N)^*.
\end{split}
\end{equation*}
Obviously, $(\hat W_t)_{t\ge0}$ is an $N$-dimensional Brownian
motion. Then, \eqref{eq4} can be reformulated as
\begin{equation}\label{B1}
\d X_t=\hat b(X_t)\d t+\hat\si(X_t)\d \hat W_t,~~~t\ge0.
\end{equation}
By Yamada-Watanabe theorem (see e.g. \cite{IW}), to show \eqref{B1}
has a unique strong solution, it is sufficient to verify that
\eqref{B1} possesses a weak solution and that it is pathwise unique.
By \eqref{b1}, \eqref{EQ1} and ({\bf H2}), a straightforward
calculation shows that
\begin{equation}\label{B2}
\begin{split}
|\hat b(x)|+\|\hat\si(x)\|_{\rm HS}\le
\left(\sum_{i=1}^N|b(x_i,\tt\mu^N_x)|^2\right)^{\frac{1}{2}}+\Big(\sum_{i=1}^N\si(x_i)^2\Big)^{1/2}\le
C_{N}(1+|x|),~~~x\in\R^N
\end{split}
\end{equation}
for some constant $C_{N}>0$,  that is, both $\hat b$ and $\hat\si$
are at most of linear growth. Observe that
\begin{equation}\label{W4}\ff{1}{N}\sum_{j=1}^N(\dd_{x_j}\times\dd_{y_j})\in\mathcal
{C}(\tt\mu^N_x,\tt\mu^N_y),~~~~x_j,y_j\in\R,\end{equation} so that we
have
\begin{equation}\label{W3}
\mathbb{W}_1(\tt\mu^N_x,\tt\mu^N_y)\le
\ff{1}{N}\sum_{j=1}^N|x_j-y_j|.
\end{equation}
This, together with \eqref{b1} and \eqref{EQ1}, besides ({\bf H2}),
implies that
\begin{equation}\label{B3}
\begin{split}
|\hat b(x)-\hat b(x')|\le \hat C_{N}\{|x-x'|+|x-x'|^\bb\},~~~\|\hat
\si(x)-\hat \si(x')\|_{\rm HS}\le \hat C_{N}|x-x'|^\aa
\end{split}
\end{equation}
for some constant $\hat C_{N}>0$ so that $\hat b$ and $\hat\si$ are
continuous . Consequently, \eqref{B2} and \eqref{B3} yields that
\eqref{B1} enjoys a weak solution. Moreover, by carrying out a
similar  argument to derive \eqref{W11}, we can infer that
\eqref{eq4} is pathwise unique. As a result, we reach a conclusion
that \eqref{eq4} has a unique strong solution.
\end{proof}

The following lemma reveals the phenomenon upon propagation of chaos
and provides the corresponding convergence rate.

\begin{lem}\label{pro1}
Under the assumptions of Theorem \ref{th2},   for any $T>0$, there
exists a constant $C_T>0$ such that
\begin{equation}\label{W2}
\sup_{i\in\mathcal {S}_N}\E\Big(\sup_{0\le t\le
T}|X_t^i-X_t^{i,N}|\Big)\le C_T\Big\{{\bf
1}_{\{\aa=\ff{1}{2}\}}N^{-\ff{1}{8}}+{\bf
1}_{\{\aa\in(\ff{1}{2},1]\}}N^{-\ff{2\aa-1}{4}}\Big\}
\end{equation}
and
\begin{equation}\label{W7}
\sup_{i\in\mathcal {S}_N}\E\Big(\sup_{0\le t\le
T}|X_t^i-X_t^{i,N}|^2\Big)\le C_T N^{-\ff{1}{4}}.
\end{equation}
\end{lem}

\begin{proof}
In what follows, we let $i\in\mathcal {S}_N$ and set
$Z^{i,N}_t:=X_t^i-X_t^{i,N}$. First of all, we are going to claim
that there exists a constant $C_T>0$ such that
\begin{equation}\label{W1}
\sup_{0\le t\le T}\E|Z^{i,N}_t|\le C_TN^{-\ff{1}{4}}.
\end{equation}
 Applying It\^o's formula to
$V_\vv:=V_{\e^{\ff{1}{\vv}},\vv}$ yields
\begin{equation}\label{eq7}
\begin{split}
\d V_\vv(Z^{i,N}_t)&=
 \{V_\vv'(Z^{i,N}_t)(b(X_t^i,\mu_t^i)-b(X_t^{i,N},\hat\mu_t^N))\\
 &\quad+\ff{1}{2}V_\vv''(Z^{i,N}_t)(\si(X_t^i)-\si(X_t^{i,N}))^2 \}\d
t +\d M_t^{i,N},
\end{split}
\end{equation}
where
$$\d M_t^{i,N}:= V_\vv'(Z^{i,N}_t)(\si(X_t^i)-\si(X_t^{i,N}))\d W_t^i.$$
Due to $X_0^{i}=X_0^{i,N}$, we henceforth obtain from
\eqref{b1}, \eqref{EQ1}, \eqref{R1} and \eqref{R2} with
$\gg=\e^{\ff{1}{\vv}}$ that
\begin{equation}\label{w1}
|Z^{i,N}_t|\le \vv+C_1\int_0^t\Big\{
|Z^{i,N}_s|+\mathbb{W}_1(\mu_s^i,\hat\mu_s^N)+\vv^{2\aa}\Big\}\d
s+M_t^{i,N}
\end{equation}
for some constant $C_1>0 $. So, by taking $\vv\downarrow0$ and
utilizing the triangle inequality for $\mathbb{W}_1$, one obtains
that
\begin{equation*}
\E|Z^{i,N}_t|\le
C_1\int_0^t\{\E|Z^{i,N}_s|+\E\mathbb{W}_1(\mu_s^i,\tt\mu_s^N)+\E\mathbb{W}_1(\tt\mu_s^N,\hat\mu_s^N)\}\d
s,
\end{equation*}
where $\tt\mu^N$ was introduced in \eqref{H1}. In terms of
\cite[Theorem 5.8]{CD}, there exists a constant $C_2>0$ such that
\begin{equation}\label{w2}
\E\mathbb{W}_1(\mu_t^i,\tt\mu_t^N)\le C_2 N^{-1/4}.
\end{equation}
As a consequence, by exploiting \eqref{W3} and \eqref{w2}, we derive
that
\begin{equation*}
\begin{split}
\E|Z^{i,N}_t|&\le
C_1\int_0^t\Big\{\E|Z^{i,N}_s|+\ff{1}{N}\sum_{j=1}^N\E|X_s^j-X_s^{j,N}|+C_2
N^{-1/4}\Big\}\d
s\\
&\le C_3\int_0^t \{\E|Z^{i,N}_s|+   N^{-1/4}\}\d s
\end{split}
\end{equation*}
for some constant $C_3>0,$ where in the last display we used the
fact that $(Z^{j,N})_{1\le j\le N}$ are identically distributed.
Subsequently, by employing Gronwall's inequality,  \eqref{W1} is
available.

Next, by  BDG's inequality and Jensen's inequality, we derive from
\eqref{b1} and \eqref{EQ1} that there exist constants $C_4,C_5>0$
such that
\begin{equation*}
\begin{split}
\E\Big(\sup_{0\le s\le t}|Z^{i,N}_s|\Big)&\le C_4\int_0^t
\{\E|Z^{i,N}_s|+   N^{-1/4}\}\d
s+C_4\E\Big(\int_0^t|Z^{i,N}_s|^{2\aa}\d s\Big)^{1/2}\\
&\le C_4\int_0^t \{\E|Z^{i,N}_s|+   N^{-1/4}\}\d
s+C_4{\bf1}_{\{\aa=1/2\}}\Big(\int_0^t\E|Z^{i,N}_s|\d s\Big)^{1/2}\\
&\quad+{\bf1}_{\{\aa\in(1/2,1]\}}\Big\{\ff{1}{2}\E\Big(\sup_{0\le
s\le t}|Z^{i,N}_s|\Big)+C_5
 \int_0^t(\E|Z^{i,N}_s|)^{2\aa-1}\d s\Big\}.
\end{split}
\end{equation*}
As a result, \eqref{W2} follows from
\eqref{W1}.

Again, by applying H\"older's inequality and  BDG's inequality,  it
follows from ({\bf H2}) and \eqref{w1} that there exists a constant $C_6>0$ such
that
\begin{equation*}
\begin{split}
\E\Big(\sup_{0\le s\le t}|Z^{i,N}_s|^2\Big)
&\le
C_6t\int_0^t\Big\{\E|Z^{i,N}_s|^2+\E\mathbb{W}_1(\mu_s^i,\hat\mu_s^N)^2
\Big\}\d s+C_6\int_0^t\E|Z^{i,N}_s|^{2\aa}\d s.
\end{split}
\end{equation*}
Owing to \eqref{W4}, we have
\begin{equation}\label{S4}
\mathbb{W}_2\Big(\ff{1}{N}\sum_{j=1}^N\dd_{x_j},\ff{1}{N}\sum_{j=1}^N\dd_{y_j}\Big)^2\le
\ff{1}{N}\sum_{j=1}^N|x_j-y_j|^2,~~~x_j,y_j\in\R.
\end{equation}
Whence, it follows that
\begin{equation}\label{W5}
\E\mathbb{W}_2(\tt\mu_t^N,\hat\mu_t^N)^2\le \ff{1}{N}\sum_{j=1}^N\E
|Z^{j,N}_t|^2=\E |Z^{i,N}_t|^2
\end{equation}
by taking the fact that  $(Z^{j,N})_{1\le j\le N}$ are identically
distributed into consideration. Moreover, according to \cite[Theorem
5.8]{CD}, there exists a constant $C_7>0$ such that
\begin{equation}\label{W6}
\E\mathbb{W}_2(\mu_s^i,\tt\mu_s^N)^2\le C_7 N^{-1/2}.
\end{equation}
Thus, combining \eqref{W5} with \eqref{W6} and employing Young's
inequality, we infer that
\begin{equation*}
\begin{split}
\E\Big(\sup_{0\le s\le t}|Z^{i,N}_s|^2\Big)&\le
C_8t\int_0^t\Big\{\E|Z^{i,N}_s|^2+\E\mathbb{W}_2(\mu_s^i,\tt\mu_s^N)^2+\E\mathbb{W}_2(\tt\mu_s^N,\hat\mu_s^N)^2
\Big\}\d s\\
&\quad+C_6\int_0^t\E|Z^{i,N}_s|^{2\aa}\d s\\
&\le C_9t\int_0^t\{\E|Z^{i,N}_s|^2+N^{-1/2}\}\d s+{\bf
1}_{\{\aa=\ff{1}{2}\}}C_6\int_0^t\E|Z^{i,N}_s|\d s\\
&\quad+{\bf
1}_{\{\aa\in(1/2,1]\}}C_6\int_0^t\Big\{2(1-\aa)\E|Z^{i,N}_s|+(2\aa-1)\E|Z^{i,N}_s|^2
\Big\}\d
s\\
\end{split}
\end{equation*}
for some constants $C_8,C_9>0.$ Finally, \eqref{W7} holds true from \eqref{W1}.
\end{proof}

The lemma below demonstrates the convergence rate of the continuous
time EM scheme associated with \eqref{eq4}.

\begin{lem}\label{pro2}
Under the assumptions of Lemma \ref{lem}, then, for any $T>0$, there
exists a constant $C_T>0$ such that 
\begin{equation}\label{F6}
\sup_{i\in\mathcal {S}_N}\E\Big(\sup_{0\le t\le
T}|X^{i,N}_t-X^{\dd,i,N}_t|\Big)\le
 C_T\Big\{{\bf1}_{\{\aa=\ff{1}{2}\}} \Big(\ff{1}{ \ln\ff{1}{\dd}
}\Big)^{\ff{1}{2}}+{\bf1}_{\{\aa\in(\ff{1}{2},1]\}}\Big(\dd^{\ff{(2\aa-1)^2}{2}}+\dd^{\ff{\bb(2\aa-1)}{2}}\Big)\Big\},
\end{equation}
and
\begin{equation}\label{F7}
\begin{split}
\sup_{i\in\mathcal {S}_N}&\E\Big(\sup_{0\le t\le
T}|X^{i,N}_t-X^{\dd,i,N}_t|^2\Big)\\
&\le C_T\Big\{{\bf1}_{\{\aa=\ff{1}{2}\}}  \ff{1}{ \ln\ff{1}{\dd} }
+{\bf1}_{\{\aa\in(\ff{1}{2},1)\}}(\dd^{2\aa-1})+{\bf1}_{\{\aa=1\}}(\dd+\dd^\bb)\Big\}.
\end{split}
\end{equation}

\end{lem}

\begin{proof}
For $i\in\mathcal {S}_N$, let
$Z^{\dd,i,N}_t=X^{i,N}_t-X^{\dd,i,N}_t$ and
$\Lambda_t^{\dd,i,N}=X^{\dd,i,N}_t-X^{\dd,i,N}_{t_\dd}$. By using
H\"older's inequality and BDG's inequality, for any $q>0$, we obtain
from \eqref{W0} that there exists $\hat C_{T,q}>0$ such that
\begin{equation}\label{W00}
\sup_{0\le t\le T}\E|\Lambda_t^{\dd,i,N}|^q\le \hat
C_{T,q}\dd^{q/2}.
\end{equation}
Below, by It\^o's formula, it follows that
\begin{equation*}
\begin{split}
\d V_{\gamma,\vv}(Z^{\dd,i,N}_t)&=
 \{V_{\gamma,\vv}'(Z^{\dd,i,N}_t)(b(X_t^{i,N},\hat\mu_t^N)-b(X_{t_\dd}^{\dd,i,N},\hat\mu_{t_\dd}^{\dd,N}))\\
 &\quad+\ff{1}{2}V_{\gamma,\vv}''(Z^{\dd,i,N}_t)(\si(X_t^{i,N})-\si(X_{t_\dd}^{\dd,i,N}))^2\}\d
t +\d \hat M_t^{i,N},
\end{split}
\end{equation*}
where
\begin{equation*}
\d \hat
M_t^{i,N}:=V_{\gamma,\vv}'(Z^{\dd,i,N}_t)(\si(X_t^{i,N})-\si(X_{t_\dd}^{\dd,i,N}))\d
W_t^i.
\end{equation*}
Then, combining ({\bf H1}) with ({\bf H2}) and taking advantage of
\eqref{R1}, \eqref{R2} as well as \eqref{W00} gives
%
\begin{equation*}\label{F1}
\begin{split}
\E|Z^{\dd,i,N}_t|&\le \varepsilon+ c_1\int_0^t\E\Big\{|Z_s^{\dd,i,N}|+|\Lambda_s^{\dd,i,N}|+|\Lambda_s^{\dd,i,N}|^\bb+\mathbb{W}_1(\hat\mu_t^N,\hat\mu_{t_\dd}^{\dd,N})\\
&\quad+\ff{1}{|Z_s^{\dd,i,N}|\ln\gg}{\bf1}_{[\vv/\gg,\vv]}(|Z_s^{\dd,i,N}|)(|Z_s^{\dd,i,N}|^{2\aa}+|\Lambda_s^{\dd,i,N}|^{2\aa})\Big\}\d s\\
&\le C_{1,T}\Big\{ \vv+\ff{\vv^{2\aa-1}}{\ln\gg} +\ff{
\gg}{\vv\ln\gg}\dd^{\aa}+ \dd^{\ff{1}{2}}+\dd^{\ff{\bb}{2}}+\int_0^t
\E|Z_s^{\dd,i,N}| \d s\Big\}
\end{split}
\end{equation*}
for some constants  $c_1,C_{1,T}>0$, where we  also utilized
\begin{equation*}
\E\mathbb{W}_1(\hat\mu_t^N,\hat\mu_{t_\dd}^{\dd,N})\le
\E|\Lambda_t^{\dd,i,N}|+\E|Z_t^{\dd,i,N}|.
\end{equation*}
Thus, Gronwall's inequality yields
\begin{equation}
\begin{split}
\E|Z^{\dd,i,N}_t|\le C_{2,T}\Big\{ \vv+\ff{\vv^{2\aa-1}}{\ln\gg}
+\ff{ \gg}{\vv\ln\gg}\dd^{\aa}+
\dd^{\ff{1}{2}}+\dd^{\ff{\bb}{2}}\Big\}
\end{split}
\end{equation}
for some constant $C_{2,T}>0.$ Furthermore, by virtue of BDG's
equality and Jensen's inequality, we deduce from ({\bf H1}), ({\bf
H2}), \eqref{R1}, and \eqref{R2} that
\begin{equation}\label{F2}
\begin{split}
\E\Big(\sup_{0\le s\le t}|Z^{\dd,i,N}_s|\Big)&\le C_{2,T}\Big\{
\vv+\ff{\vv^{2\aa-1}}{\ln\gg} +\ff{ \gg}{\vv\ln\gg}\dd^{\aa}+
\dd^{\ff{1}{2}}+\dd^{\ff{\bb}{2}}\Big\}\\
&\quad+{\bf1}_{\{\aa=1/2\}}c_1\Big(\int_0^t(\E|Z^{\dd,i,N}_s|+\E|\Lambda^{\dd,i,N}_{s}|)\d s\Big)^{1/2}\\
&\quad+{\bf1}_{\{\aa\in(1/2,1]\}}\Big\{\ff{1}{2}\E\Big(\sup_{0\le
s\le t}|Z^{\dd,i,N}_s|\Big)+c_2
 \int_0^t(\E|Z^{\dd,i,N}_s|)^{2\aa-1}\d s\\
 &\quad+c_2\left(\int_0^t\E|\Lambda^{\dd,i,N}_s|^{2\aa}\d s\right)^{\frac{1}{2}}\Big\}
\end{split}
\end{equation}
and that
\begin{equation}\label{F5}
\begin{split}
\E\Big(\sup_{0\le s\le t}|Z^{\dd,i,N}_s|^2\Big)&\le C_{3,T}\Big\{
\vv+\ff{\vv^{2\aa-1}}{\ln\gg} +\ff{ \gg}{\vv\ln\gg}\dd^{\aa}+
\dd^{\ff{1}{2}}+\dd^{\ff{\bb}{2}}\Big\}^2\\
&\quad+{\bf1}_{\{\aa=1/2\}} c_3\int_0^t(\E|Z^{\dd,i,N}_s|+\E|\Lambda^{\dd,i,N}_{s}|)\d s \\
&\quad+{\bf1}_{\{\aa\in(1/2,1]\}}c_4\Big\{
 \int_0^t(\E|Z^{\dd,i,N}_s|+\E|Z^{\dd,i,N}_s|^2) \d s\\
 &\quad+\Big(\int_0^t\E|\Lambda^{\dd,i,N}_s|^{2\aa} \d
 s\Big)^{\ff{1}{2}}\Big\}.
\end{split}
\end{equation}
Consequently, the desired assertions \eqref{F6} and \eqref{F7}
follows from \eqref{F2} and \eqref{F5} and by taking
$\vv=\ff{1}{\ln\ff{1}{\dd}}$ and $\gg=(1/\dd)^{\ff{1}{3}}$ and
$\vv=\ss\dd$ and $\gg=2$, respectively, for $\aa=\ff{1}{2}$ and
$\aa\in(1/2,1].$
\end{proof}

Now with the help of Lemmas \ref{pro1} and Lemma \ref{pro2}, we  complete
directly the proof of Theorem \ref{th2}.

\section{Proof of Theorem \ref{th3}}

The proof of Theorem \ref{th3} is based on two  lemmas below, where
the first one is concerned with  propagation of chaos for
McKean-Vlasov SDEs with irregular drift coefficients.

\begin{lem}\label{pro3}
Under the assumptions of Theorem \ref{th3},  for any $T>0$, there
exists a constant $C_T>0$ such that
\begin{equation}\label{S1}
\sup_{i\in\mathcal {S}_N}\E\Big(\sup_{0\le t\le
T}|X_t^i-X_t^{i,N}|^2\Big)\le C_T
\begin{cases}
N^{-\ff{1}{2}},~~~~~~~~~~~~~d<4\\
N^{-\ff{1}{2}}\log N,~~~~ ~~d=4\\
N^{-\ff{2}{d}},~~~~~~~~~~~~~ d>4.
\end{cases}
\end{equation}
\end{lem}

\begin{proof}
For any $i\in\mathcal {S}_N$ and $x\in\R^d$, let
$b_t^{\mu^i}(x)=b(x,\mu_t^i)$ and
$b_t^{\hat\mu^N}=b(x,\hat\mu_t^N)$. Then, \eqref{C4} and \eqref{eq4}
can be rewritten respectively  as
\begin{equation*}
\begin{split}
\d X_t^i&=b_t^{\mu^i}(X_t^i)\d t+\si(X_t^i)\d W_t^i\\
\d X_t^{i,N}&=b_t^{\hat\mu^N}(X_t^{i,N})\d t+\si(X_t^{i,N})\d W_t^i.
\end{split}
\end{equation*}
For  $\ll>0 $ and $\mu^i_\cdot\in C([0,T];\mathscr{P}(\R^d))$,
consider the following PDE for
$u^{\ll,\mu^i}:[0,T]\times\R^d\to\R^d$:
\begin{equation}\label{A1}
\partial_tu^{\ll,\mu^i}_t+\ff{1}{2}\sum_{k,j=1}^d\<\si\si^*\e_j,e_l\>\nn_{e_j}\nn_{e_l}u^{\ll,\mu^i}_t+\nn_{b_t^{\mu^i}}u^{\ll,\mu^i}_t+
b_t^{\mu^i}=\ll u^{\ll,\mu^i}_t,~~~u^{\ll,\mu^i}_T=0,
\end{equation}
where $(e_j)_{1\le j\le d}$ stands for the orthogonal basis of
$\R^d.$ By recurring to  \cite[Lemma 2.1]{BHY}, for $\ll>0$  large
enough, \eqref{A1} has a unique solution $u^{\ll,\mu^i} $ with
\begin{equation}\label{A2}
\|\nn u^{\ll,\mu^i}\|_\8+\|\nn^{ 2 } u^{\ll,\mu^i}\|_\8\le
\ff{1}{2},
\end{equation}
where $\|\cdot\|_\8$ means the uniform norm. Applying It\^o's
formula to $\theta^{\ll,\mu^i}_t(x):=x+u^{\ll,\mu^i}_t(x),
x\in\R^d,$ where $u^{\ll,\mu^i}$ solves \eqref{A1},    yields
\begin{equation}\label{S2}
\begin{split}
\d \theta^{\ll,\mu^i}_t(X_t^i)&=\ll u^{\ll,\mu^i}_t(X_t^i)\d
t+(\nn\theta^{\ll,\mu^i}_t\si)(X_t^i)\d W_t^i\\
\d
\theta^{\ll,\mu^i}_t(X_t^{i,N})
 &=\{\ll
u^{\ll,\mu^i}_t(X_t^{i,N})+\nn\theta^{\ll,\mu^i}_t(b_t^{\hat\mu^N}-b_t^{\mu^i})(X_t^{i,N})\}\d
t+(\nn\theta^{\ll,\mu^i}_t\si)(X_t^{i,N})\d W_t^i.
 \end{split}
\end{equation}
Henceforth, for
$\Lambda^{\ll,i,N}_t:=\theta^{\ll,\mu^i}_t(X_t^i)-\theta^{\ll,\mu^i}_t(X_t^{i,N})$,
we derive from  H\"older's inequality and BDG's inequality that
\begin{equation*}
\begin{split}
\E\Big(\sup_{0\le s\le t}|\Lambda^{\ll,i,N}_s|^2\Big)&\le
C_{1,\ll}t\Big\{\int_0^t\E |u^{\ll,\mu^i}_s(X_s^i)-u^{\ll,\mu^i}_s(X_s^{i,N})|^2\d s\\
&\quad+\int_0^t\E|(\nn\theta^{\ll,\mu^i}_s(b_s^{\hat\mu^N}-b_s^{\mu^i}))(X_s^{i,N})|^2\d s\Big\}\\
&\quad+\int_0^t\E\|(\nn\theta^{\ll,\mu^i}_s\si)(X_s^{i,N})-(\nn\theta^{\ll,\mu^i}_s\si)(X_s^{i})\|_{\rm
HS}^2 \d s\\
&=:C_{1,\ll}t\{I_{1,i}(t)+I_{2,i}(t)\}+I_{3,i}(t)
\end{split}
\end{equation*}
for some constant $C_{1,\ll}>0.$ Set $Z_t^{i,N}:=X_t^i-X_t^{i,N}$
for convenience. By means of \eqref{A2}, one has
\begin{equation}\label{A3}
I_{1,i}(t)\le C_1\int_0^t\E|Z_s^{i,N}|^2\d s.
\end{equation}
for some constant $C_1>0$. Next, via \eqref{C1}  and \eqref{A2}, in
addition to  \eqref{W5}, it follows from the triangle inequality
that
\begin{equation}\label{A4}
\begin{split}
I_{2,i}(t)&\le
C_2\int_0^t\{\E\mathbb{W}_2(\hat\mu^N_s,\tt\mu^N_s)^2+\E\mathbb{W}_2(\tt\mu^N_s,\mu^i_s)^2\}\d
s\\
&\le
C_2\int_0^t\{\E|Z_s^{i,N}|^2+\E\mathbb{W}_2(\tt\mu^N_s,\mu^i_s)^2\}\d
s
\end{split}
\end{equation}
for some constant $C_2>0.$ Furthermore, owing to \eqref{A5} and
\eqref{A2}, we obtain that for some constant $C_3>0,$
\begin{equation}\label{A55}
\begin{split}
I_{3,i}(t)&\le2\int_0^t\E\|
\nn\theta^{\ll,\mu^i}_s(X_s^{i,N})(\si(X_s^{i,N})-\si
(X_s^{i}))\|_{\rm HS}^2 \d s\\
&\quad+2\int_0^t\E\|(\nn\theta^{\ll,\mu^i}_s(X_s^{i,N})-\nn\theta^{\ll,\mu^i}_s(X_s^{i}))\si
(X_s^{i})\|_{\rm HS}^2\d s\\
&\le C_3\int_0^t\E|Z_s^{i,N}|^2\d s.
\end{split}
\end{equation}
Thus, with the aid of \eqref{A3}, \eqref{A4} and \eqref{A55}, we
find that for some constant $C_{2,\ll}>0,$
\begin{equation*}
\E\Big(\sup_{0\le s\le t}|\Lambda^{\ll,i,N}_s|^2\Big)\le
C_{2,\ll}(t+1)\int_0^t\{\E|Z_s^{i,N}|^2+\E\mathbb{W}_2(\tt\mu^N_s,\mu^i_s)^2\}\d
s.
\end{equation*}
This, together with the facts that $ |Z_t^{i,N}|^2\le
4|\Lambda^{\ll,i,N}_t|^2 $ due to \eqref{A2},  leads to
\begin{equation*}
\E\Big(\sup_{0\le s\le t}|Z_s^{i,N}|^2\Big)\le
C_{3,\ll}(t+1)\int_0^t\{\E|Z_s^{i,N}|^2+\E\mathbb{W}_2(\tt\mu^N_s,\mu^i_s)^2\}\d
s
\end{equation*}
for some constant $C_{3,\ll}>0$. Hence, the desired assertion
\eqref{S1} follows from Gronwall's inequality and the fact that
\begin{equation}\label{G1}\sup_{0\le t\le
T}\E\mathbb{W}_2(\tt\mu^N_t,\mu^i_t)^2\leq C_4
\begin{cases}
N^{-\frac{1}{2}}, ~~~~~~~~~~d<4 \\
N^{-\frac{1}{2}}\log N, ~~~  d=4
 \\
 N^{-\frac{2}{d}},
  ~~~~~~~~~~~d>4
 \end{cases}
\end{equation}
for some constant $C_4>0$; see, for instance, \cite[Theorem
5.8]{CD}.
\end{proof}

\begin{lem}\label{D1}
Under the assumptions of Theorem \ref{th3},  for any $T>0$, there
exists a constant $C_T>0$ such that
\begin{equation}\label{S6}
\sup_{i\in\mathcal {S}_N}\E\Big(\sup_{0\le t\le
T}|X_t^{i,N}-X_t^{\dd,i,N}|^2\Big)\le C_T
\begin{cases}
\dd^\aa+N^{-\ff{1}{2}},~~~~~~~~~~~~~d<4\\
\dd^\aa+N^{-\ff{1}{2}}\log N,~~~~ ~~d=4\\
\dd^\aa+N^{-\ff{2}{d}},~~~~~~~~~~~~~ d>4.
\end{cases}
\end{equation}

\end{lem}

\begin{proof}
Below we let $t\in[0,T]$.  For $x\in\R^d$ and $i\in\mathcal {S}_N$,
let $b^{\hat\mu^{\dd,N}}_{k\dd}(x)=b(x,\hat\mu^{\dd,N}_{k\dd})$ so
that \eqref{C3} can be reformulated as
\begin{equation*}
\d X_t^{\dd,i,N}=b^{\hat\mu^{\dd,N}}_{t_\dd}(X_{t_\dd}^{\dd,i,N})\d
t+\si(X_{t_\dd}^{\dd,i,N})\d W_t^i.
\end{equation*}
Applying It\^o's formula to
$\theta^{\ll,\mu^i}_t(x)=x+u^{\ll,\mu^i}_t(x)$ and taking the fact
that $u^{\ll,\mu^i}$ solves \eqref{A1} into consideration gives that
\begin{equation}\label{S3}
\begin{split}
\d \theta^{\ll,\mu^i}_t( X_t^{\dd,i,N})&=\Big\{\ll u^{\ll,\mu^i}_t(
X_t^{\dd,i,N})  +\nn\theta^{\ll,\mu^i}_t(
X_t^{\dd,i,N})(b^{\hat\mu^{\dd,N}}_{t_\dd}(X_{t_\dd}^{\dd,i,N})-b^{ \mu^i}_t(X_t^{\dd,i,N}))\\
&\quad
+\ff{1}{2}\sum_{k,j=1}^d\<((\si\si^*)(X_{t_\dd}^{\dd,i,N})-(\si\si^*)(X_t^{\dd,i,N}))\e_j,e_l\>\nn_{e_j}\nn_{e_l}u^{\ll,\mu^i}_t(
X_t^{\dd,i,N})\Big\}\d t\\
&\quad+\nn\theta^{\ll,\mu^i}_t(
X_t^{\dd,i,N})\si(X_{t_\dd}^{\dd,i,N})\d W_t^i.
\end{split}
\end{equation}
Set
\begin{equation*}
\Theta^{\ll,i,N}_t:=\theta^{\ll,\mu^i}_t(X_t^{i,N})-\theta^{\ll,\mu^i}_t(X_t^{\dd,i,N}),~~~~
 Z_t^{\dd,i,N}:=X_t^{i,N}-X_t^{\dd,i,N}.
\end{equation*}
Then,  from \eqref{S3} and    the second SDE in \eqref{S2}, we
deduce from H\"older's inequality and BDG's inequality   that
\begin{equation*}
\begin{split}
&\E\Big(\sup_{0\le s\le t} |\Theta^{\ll,i,N}_s|^2\Big)\\&\le
C_{1,\ll,d}(1+t)\Big\{\int_0^t\E|u^{\ll,\mu^i}_s(
X_s^{i,N})-u^{\ll,\mu^i}_s( X_s^{\dd,i,N}) |^2\d s\\
&\quad+\int_0^t\E|\nn\theta^{\ll,\mu^i}_s(b_s^{\hat\mu^N}-b_s^{\mu^i})(X_s^{i,N})-\nn\theta^{\ll,\mu^i}_s(
X_s^{\dd,i,N})(b^{\hat\mu^{\dd,N}}_{s_\dd}(X_{s_\dd}^{\dd,i,N})-b^{
\mu^i}_s(X_s^{\dd,i,N}))|^2\d s\\
&\quad+\sum_{k,j=1}^d\int_0^t
\E|\<((\si\si^*)(X_{s_\dd}^{\dd,i,N})-(\si\si^*)(X_s^{\dd,i,N}))\e_j,e_l\>\nn_{e_j}\nn_{e_l}u^{\ll,\mu^i}_s(
X_s^{\dd,i,N})|^2\d s\\
&\quad+\int_0^t\|(\nn\theta^{\ll,\mu^i}_s\si)(X_s^{i,N})-\nn\theta^{\ll,\mu^i}_s(
X_s^{\dd,i,N})\si(X_{s_\dd}^{\dd,i,N})\|_{\rm HS}^2\d s\Big\}\\
&=C_{\ll,d}(1+t)\{J_1(t)+J_2(t)+J_3(t)+J_4(t)\}
\end{split}
\end{equation*}
for some constant $C_{\ll,d}>0.$ In what follows, we intend to estimate
$J_i(t), i=1,2,3,4$, one-by-one. Owing to \eqref{A2}, there exists a
constant $c_1>0$ such that
\begin{equation}\label{L1}
J_1(t)\le c_1\int_0^t\E |Z_s^{\dd,i,N}|^2\d s.
\end{equation}
Next, thanks to \eqref{C1} and \eqref{A2}, it follows from
\eqref{S4} that
\begin{equation}\label{A10}
\begin{split}
J_2(t)&\le
c_2\int_0^t\{\E\mathbb{W}_2(\mu_s^i,\hat\mu^N_s)^2+\E|X_s^{\dd,i,N}-X_{s_\dd}^{\dd,i,N}|^{2\aa}+\E\mathbb{W}_2(\mu_s^i,\hat\mu^{\dd,N}_{s_\dd})^2\}\d
s\\
&\le
c_3\int_0^t\{\dd^\aa+\E\mathbb{W}_2(\mu_s^i,\tt\mu^N_s)^2+\E\mathbb{W}_2(\tt\mu^N_s,\hat\mu^N_s)^2
  +\E\mathbb{W}_2(\tt\mu^N_s,\hat\mu^{\dd,N}_{s_\dd})^2\}\d
s\\
\end{split}
\end{equation}
for some constants $c_2,c_3>0$, where we have used, for some
constant $c_{q,T}>0$,
\begin{equation}\label{S5}
\sup_{0\le t\le T}\E|X_t^{\dd,i,N}-X_{t_\dd}^{\dd,i,N}|^q\le
c_{q,T}\dd^{q/2},~~~~q>0,
\end{equation}
which can be obtained in a standard way under the assumption
\eqref{A5}. On the other hand,  by virtue of \eqref{S4} and
\eqref{S5}, we have for some $C_{1,T}>0,$
\begin{equation}\label{H2}
\begin{split}
\E\mathbb{W}_2(\tt\mu^N_t,\hat\mu^N_t)^2
  +\E\mathbb{W}_2(\tt\mu^N_t,\hat\mu^{\dd,N}_{t_\dd})^2&\le\ff{1}{N}
  \sum_{j=1}^N\{\E|X_t^j-X_t^{j,N}|^2+\E|X_t^j-X_{t_\dd}^{\dd,j,N}|^2\}\\
  &\le C_{1,T}\dd+\E|X_t^i-X_t^{i,N}|^2+2\E|X_t^i-X_t^{\dd,i,N}|^2,
\end{split}
\end{equation}
where in the last display we used the facts that
$(X^j-X^{j,N})_{j\in\mathcal {S}_N}$ and
$(X^j-X^{\dd,j,N})_{j\in\mathcal {S}_N}$ are identically
distributed. Then, plugging \eqref{H2} back into \eqref{A10} gives
that
\begin{equation}\label{L2}
J_2(t) \le
C_{2,T}\int_0^t\{\dd^\aa+\E\mathbb{W}_2(\mu_s^i,\tt\mu^N_s)^2+\E|X_s^i-X_s^{i,N}|^2+
\E|Z_s^{\dd,i,N}|^2\}\d s
\end{equation}
for some constant $C_{2,T}>0$. Furthermore, taking \eqref{A5},
\eqref{A2}, and \eqref{S5} into account, we find that there exists a
constant  $C_{3,T}>0$ such that
\begin{equation}\label{L3}
J_3(t)+J_4(t) \le C_{3,T}\int_0^t\{\dd+\E|Z_s^{\dd,i,N}|^2 \}\d s.
\end{equation}
Now, combining \eqref{L1}, \eqref{L2} with  \eqref{L3}, we arrive at
\begin{equation*}
\E\Big(\sup_{0\le s\le t} |\Theta^{\ll,i,N}_s|^2\Big)\le C_{4,T}
\int_0^t\{\dd^\aa+\E\mathbb{W}_2(\mu_s^i,\tt\mu^N_s)^2+\E|X_s^i-X_s^{i,N}|^2+
\E|Z_s^{\dd,i,N}|^2\}\d s
\end{equation*}
for some constant $C_{4,T}>0.$ This, together with $
|Z_t^{\dd,i,N}|^2\le 4|\Theta^{\ll,i,N}_t|^2 $ due to \eqref{A2},
yields
\begin{equation*}
\E\Big(\sup_{0\le s\le t} |Z_s^{\dd,i,N}|^2\Big)\le C_{5,T}
\int_0^t\{\dd^\aa+\E\mathbb{W}_2(\mu_s^i,\tt\mu^N_s)^2+\E|X_s^i-X_s^{i,N}|^2+
\E|Z_s^{\dd,i,N}|^2\}\d s
\end{equation*}
for some constant $C_{5,T>0}$. Consequently, the desired assertion
\eqref{S6} holds true by  applying Gronwall's inequality and
employing \eqref{S1} and \eqref{G1}.
\end{proof}

Now on the basis of Lemmas \ref{pro3} and \eqref{D1}, the proof of
Theorem \ref{th3} can be complete.

\end{document}